\documentclass{amsart}

\usepackage{amssymb}
\usepackage{color}
\usepackage[allcolors = blue, colorlinks = true]{hyperref}
\usepackage{amsrefs}
\usepackage{mathptmx}
\usepackage{comment}
\usepackage{esint}

\newtheorem{theorem}{Theorem}[section]
\newtheorem{lemma}[theorem]{Lemma}
\newtheorem{corollary}[theorem]{Corollary}

\numberwithin{equation}{section}
\newcommand{\R}{\mathbb{R}}
\newcommand{\Rn}{\R^n}
\newcommand{\il}{\Delta_{\infty}}
\newcommand{\ilN}{\Delta_{\infty}^N}
\newcommand{\ue}{u^{\epsilon}}
\newcommand{\Ve}{V^{\epsilon}}
\newcommand{\ze}{z^{\epsilon}}
\newcommand{\la}{\langle}
\newcommand{\ra}{\rangle}
\newcommand{\CcU}{C^{\infty}_0(U)}
\newcommand{\CcO}{C^{\infty}_0(\Omega)}
\newcommand{\Om}{\Omega}
\newcommand{\loc}{\textnormal{loc}}

\DeclareMathOperator{\tr}{tr\,}
\DeclareMathOperator{\diverg}{div}
\DeclareMathOperator{\spt}{spt\,}
\DeclareMathOperator{\osc}{osc\,}
\DeclareMathOperator{\dist}{dist\,}
\begin{document}
\title[An inequality and its application to the regularity of $p$-harmonic functions]{Note on an elementary inequality and its application to the regularity of $p$-harmonic functions}
\author{Saara Sarsa}
\address{Matematiikan ja tilastotieteen laitos, Helsingin yliopisto, PL 68 (Pietari Kalmin katu 5) 00014 Helsingin yliopisto, Helsinki, Finland}
\email[Saara Sarsa]{saara.sarsa@helsinki.fi}
\thanks{S. Sarsa is supported by the Academy of Finland, the Centre of Excellence in Analysis and Dynamics Research and the Academy of Finland, project 308759.}
\begin{abstract}
We study the Sobolev regularity of $p$-harmonic functions. We show that $|Du|^{\frac{p-2+s}{2}}Du$ belongs to the Sobolev space $W^{1,2}_{\loc}$, $s>-1-\frac{p-1}{n-1}$, for any $p$-harmonic function $u$. The proof is based on an elementary inequality. 
\end{abstract}
\maketitle
\section{Introduction}

In \cite{Dong2020} Dong, Fa, Zhang and Zhou established the following inequality. Let $v$ be a smooth real-valued function defined on a domain $\Om\subset\Rn$, $n\geq2$. Let $Dv:=(v_{x_1},\ldots,v_{x_n})$ denote its gradient and $D^2v:=(v_{x_ix_j})_{i,j=1}^n$ its Hessian. The Laplacian of $v$ is denoted as 
$$ \Delta v:=\tr (D^2v)=\sum_{i=1}^nv_{x_ix_i} $$ 
and the infinity Laplacian of $v$ as 
$$ \il v:=\la Dv,D^2vDv\ra=\sum_{i,j=1}^nv_{x_i}v_{x_ix_j}v_{x_j}. $$
Then
\begin{equation} \label{eq:FundamentalInequalityinDong2019}
\Big| |D^2vDv|^2-\Delta v\il v -\frac{1}{2}\big(|D^2v|^2-(\Delta v)^2\Big)|Dv|^2\Big|
\leq \frac{n-2}{2}\big(|D^2v|^2|Dv|^2-|D^2vDv|^2\big)
\end{equation}
holds everywhere in $\Om$. The authors derived \eqref{eq:FundamentalInequalityinDong2019} as a direct consequence of the inequality
\begin{equation} \label{eq:VectorInequalityinDong2019}
\begin{aligned}
\Big| \sum_{i=1}^n(\lambda_ia_i)^2
-\Big(\sum_{i=1}^n\lambda_i\Big)\Big(\sum_{i=1}^n\lambda_ia_i^2\Big) -\frac{1}{2}\Big(|\lambda|^2-\Big(\sum_{i=1}^n\lambda_i\Big)^2\Big)\Big| \\
\leq \frac{n-2}{2}\Big(|\lambda|^2-\sum_{i=1}^n(\lambda_ia_i)^2\Big)
\end{aligned}
\end{equation}
that holds for any vectors $\lambda=(\lambda_1,\ldots,\lambda_n)\in\Rn$ and $a=(a_1,\ldots,a_n)\in\Rn$ such that $|a|=1$. For the proof of \eqref{eq:VectorInequalityinDong2019}, see the proof of Lemma 2.2 in \cite{Dong2020}. The inequality \eqref{eq:FundamentalInequalityinDong2019} is applied to study the regularity of solutions to $p$-Laplacian equation (see the equation \eqref{eq:pLaplacian} below) and its parabolic counterparts. For further details, we refer the reader to Theorems 1.1, 1.3 and 1.5 in \cite{Dong2020}. 

The inequality \eqref{eq:FundamentalInequalityinDong2019} in the case $n=2$ (when it is sharp) has been used to prove Sobolev regularity for planar infinity harmonic functions, see \cite{Koch2019}. See also \cite{Lindgren2020}.

In this paper we show that \eqref{eq:FundamentalInequalityinDong2019} can be derived as a consequence of another elementary inequality that has been used before by Colding \cite{Colding2012} to prove monotonicity formulas for solutions to certain elliptic partial differential equations. See for instance the proof of Theorem 2.4 in \cite{Colding2012}. This elementary inequality says that for any symmetric matrix $A\in\R^{n\times n}$ and for any vector $e\in\R$ we have
\begin{equation} \label{eq:MatrixInequalityIntro}
|e|^4|A|^2
\geq 2|e|^2|Ae|^2+\frac{\big(|e|^2\tr(A)-\la e,Ae\ra\big)^2}{n-1}
-\la e,Ae\ra^2.
\end{equation}
If $n=2$, we have equality instead of inequality in \eqref{eq:MatrixInequalityIntro}.

For a smooth function $v$, we apply the inequality \eqref{eq:MatrixInequalityIntro} with $A=D^2v$ and $e=Dv$ to obtain a lower bound for the Hilbert-Schmidt norm of the Hessian $D^2v$ with respect to the gradient $Dv$. More precisely, we obtain
\begin{equation} \label{eq:OurVersionofFundamentalInequalityIntro}
|Dv|^4|D^2v|^2
\geq 2|Dv|^2|D^2vDv|^2+\frac{\big(|Dv|^2\Delta v-\il v\big)^2}{n-1}
-(\il v)^2.
\end{equation}
The main point is that \eqref{eq:OurVersionofFundamentalInequalityIntro} implies \eqref{eq:FundamentalInequalityinDong2019} but not vice versa, apart from the case $n=2$ where both inequalities reduce to equality.  See Section 2 for details.
Consequently, we are able to improve Theorem 1.1 in \cite{Dong2020}, which concerns regularity of $p$-harmonic functions. 

Let $1<p<\infty$. A function $u\in W^{1,p}(\Om)$ is called $p$-harmonic, if it solves the $p$-Laplacian equation
\begin{equation} \label{eq:pLaplacian}
\Delta_p u :=\diverg\big(|Du|^{p-2}Du\big)=0
\end{equation}
in the weak sense, that is, if
$$ \int_{\Omega}|Du|^{p-2}\la Du,D\varphi\ra dx = 0 $$
for all $\varphi\in\CcO$.

Let $u$ denote a $p$-harmonic function in $\Om\subset\Rn$, $n\geq 2$. For $s\in\R$, we define the vector field $V_s\colon\Rn\to\Rn$ as
\begin{equation}
V_s(z):=
    \begin{cases}
    |z|^{\frac{p-2+s}{2}}z \quad &\text{for }z\in \Rn\setminus\{0\}; \\
    0 \quad&\text{for }z=0.
    \end{cases}
\end{equation}
We study the Sobolev regularity of the vector field $V_s(Du)\colon\Om\to\Rn$.
The letter $V$ refers to the notation used in \cites{Hamburger1992,Mingione2007,Avelin2018}.
The subscript $s$ is a perturbation parameter that describes the deviation from the ''natural'' vector field $V(Du):=V_0(Du)$. We may call the vector field $V(Du)$ ''natural'' in this setting, because its Sobolev regularity arises more naturally than the one of the gradient $Du$ alone. See for instance Proposition 2 in \cite{Bojarski1987}, where the authors apply the difference quotient characterization of Sobolev functions to show that $V(Du)\in W^{1,2}_{\loc}(\Om)$. 
For similar results, see for instance \cite{Uhlenbeck1977}*{Lemma 3.1}, \cite{Giusti2003}*{Remark 8.4} and \cite{Mingione2007}*{Lemma 3.2}.

In fact, on the contrary to the $W^{1,2}_{\loc}$-regularity of $V(Du)$, it is not certain if the weak Hessian $D^2u$ necessarily exists. Manfredi and Weitsman have shown in \cite{Manfredi1988}*{Lemma 5.1} that $p$-harmonic functions belong to $W^{2,2}_{\loc}$, provided that $1<p<3+\frac{2}{n-2}$. This restriction for the range of $p$ arises from so-called Cordes condition \cite{Cordes1961}.

In this paper we are interested in the $W^{1,2}_{\loc}$-regularity of $V_s(Du)$ when $s\neq 0$. Dong, Fa, Zhang and Zhou apply \eqref{eq:FundamentalInequalityinDong2019} to prove that $V_s(Du)\in W^{1,2}_{\loc}$ whenever
\begin{equation} \label{eq:RestrictioninDong2019}
   s > 2 - \min\Big\{p+\frac{n}{n-1},3+\frac{p-1}{n-1}\Big\}, 
\end{equation}
see \cite{Dong2020}*{Theorem 1.1}. We improve this bound to
\begin{equation} \label{eq:RestrictionOurs}
   s > -1-\frac{p-1}{n-1}. 
\end{equation}
In other words, we show that the condition $s>2-p-\frac{n}{n-1}$ is redundant and obtain nontrivial improvement in the case $1<p<2$ and $n\geq 3$.

The following theorem is an application of \eqref{eq:OurVersionofFundamentalInequalityIntro} and the main result of this paper. In the statement of the theorem, and throughout the paper, a generic ball in $\Rn$ with radius $r>0$ is denoted briefly as $B_r$.

\begin{theorem} \label{thm:OscillationEstimateforDVs}
Let $n\geq 2$, $1<p<\infty$, and $s>-1-\frac{p-1}{n-1}$. 
If $u$ is $p$-harmonic in $\Omega\subset\Rn$, then $V_s(Du)\in W^{1,2}_{\loc}(\Omega)$. 
Moreover, there exists a constant $C=C(n,p,s)>0$ such that
\begin{equation} \label{eq:OscillationEstimateforDVs}
\int_{B_r}|D(V_s(Du))|^2dx
\leq \frac{C}{r^2}\int_{B_{2r}}|V_s(Du)-z|^2dx
\end{equation}
for all vectors $z\in\Rn$ and all concentric balls $B_r\subset B_{2r}\subset\subset \Omega$.
\end{theorem}

Proof of Theorem \ref{thm:OscillationEstimateforDVs} follows from establishing the case $z=0$ in Section 3 and applying known results of $p$-harmonic functions in Section 4. Note that the right hand side of \eqref{eq:OscillationEstimateforDVs} is finite due to the well-known $C^{1,\alpha}_{\loc}$-regularity of $p$-harmonic functions for some $\alpha=\alpha(n,p)\in(0,1)$. For this classical result, we refer the reader to \cites{Uraltseva1968,Uhlenbeck1977,Evans1982,DiBenedetto1983,Lewis1983,Tolksdorf1984}. For results concerning optimal regularity of $p$-harmonic functions, see \cite{Lewis1980} and \cites{Bojarski1987,Iwaniec1989}.

Using Sobolev-Poincar\'e inequality and Gehring's Lemma \cite{Gehring1973} with the estimate \eqref{thm:OscillationEstimateforDVs} leads to a higher integrability result for $D(V_s(Du))$. Here and subsequently, we denote the integral average of a locally integrable function $v$ as
$$ (v)_{B_r}:=\fint_{B_r}vdx=\frac{1}{|B_r|}\int_{B_r}vdx. $$

\begin{corollary} \label{cor:CorollaryduetoGehring}
Under the same hypothesis as Theorem \ref{thm:OscillationEstimateforDVs}, 
there exists a constant $\delta=\delta(n,p,s)>0$ such that $D(V_s(Du))\in L^q_{\loc}(\Om)$ for every $1\leq q<2+\delta$. Moreover, there exists a constant $C=C(n,p,s,q)>0$ such that
\begin{equation} \label{eq:CorollaryduetoGehring}
\Big(\fint_{B_r}|D(V_s(Du))|^qdx\Big)^{1/q}
\leq C\Big(\fint_{B_{2r}}|D(V_s(Du))|^2dx\Big)^{1/2}
\end{equation}
for all concentric balls $B_r\subset B_{2r}\subset\subset \Omega$.
\end{corollary}

\begin{proof}
Combination of Sobolev-Poincare inequality and \eqref{eq:OscillationEstimateforDVs} with $z=\big(V_s(Du)\big)_{B_{2r}}$ yields
\begin{align*}
\Big(\fint_{B_r}|D(V_s(Du))|^2dx\Big)^{1/2}
&\leq \frac{C}{r}\Big(\fint_{B_{2r}}|V_s(Du)-\big(V_s(Du)\big)_{B_{2r}}|^2dx\Big)^{1/2} \\
&\leq C\Big(\fint_{B_{2r}}|D(V_s(Du))|^{\frac{2n}{n+2}}dx\Big)^{\frac{n+2}{2n}}.
\end{align*}
Now Gehring's Lemma is applicable. The estimate \eqref{eq:CorollaryduetoGehring} follows immediately.
\end{proof}

We finish the introduction by mentioning some interesting values of the parameter $s$. If $1<p<3+\frac{n}{n-2}$, then we can select $s=2-p$. This reproves the $W^{2,2}_{\loc}$-regularity of $p$-harmonic functions discussed above. The same conclusion can be drawn also from the stronger restriction \eqref{eq:RestrictioninDong2019} due to Dong, Fa, Zhang and Zhou. Our weakening \eqref{eq:RestrictionOurs} allows us to select $s=p-2$, which reproves the known $W^{1,2}_{\loc}$-regularity of the weakly divergence free vector field $|Du|^{p-2}Du$, see \cites{Damascelli2004,Lou2008} and \cite{Avelin2018}*{Theorem 4.1}.

\section{An elementary inequality}
In this section we explain in detail how we improve the inequality \eqref{eq:FundamentalInequalityinDong2019}.

\begin{lemma} \label{lem:MatrixInequality}
Let $A\in\R^{n\times n}$, $n\geq 2$, be a symmetric matrix and $e\in\Rn$ a vector. Then we have
\begin{equation} \label{eq:MatrixInequality}
|e|^4|A|^2
\geq 2|e|^2|Ae|^2+\frac{\big(|e|^2\tr(A)-\la e,Ae\ra\big)^2}{n-1}
-\la e,Ae\ra^2.
\end{equation}
If $n=2$, equality holds in place of the inequality in \eqref{eq:MatrixInequality}.
\end{lemma}

\begin{proof}
If $e=0$, then \eqref{eq:MatrixInequality} is trivially true, thus we prove \eqref{eq:MatrixInequality} for $e\neq 0$. Since \eqref{eq:MatrixInequality} is homogeneous, we may assume without loss of generality that $|e|=1$. We fix an orthogonal coordinate system $\{e_1,\ldots,e_n\}$ in $\Rn$, such that $e_n=e$. Let $O:=(e_1,\ldots,e_n)$ be the corresponding orthogonal rotation matrix, where $e_1,\ldots,e_n$ are interpreted as column vectors.

Denote $B:=O^{\intercal}AO=\big(\la e_i,Ae_j\ra\big)_{i,j=1}^n$. Let $B_{n-1}:=(B_{ij})_{i,j=1}^{n-1}$ be the submatrix given by the first $n-1$ rows and $n-1$ columns of $B$. We may decompose
\begin{equation} \label{eq:BasicDecompositionofA}
|B|^2
=|B_{n-1}|^2+2\sum_{i=1}^{n-1}\la e_i,Ae_n\ra^2 +\la e_n,Ae_n\ra^2.
\end{equation}
Consider the submatrix $B_{n-1}$ as an element of the Hilbert space $\R^{(n-1)\times(n-1)}$ with the Hilbert-Schmidt matrix inner product.
Apply Pythagoras's theorem to obtain
\begin{equation} \label{eq:EstimateofSubmatrixofA} 
\begin{aligned} 
|B_{n-1}|^2
&=\frac{(\tr(B_{n-1}))^2}{n-1}+\Big|B_{n-1}
-\frac{\tr(B_{n-1})}{n-1}I\Big|^2 \\
&\geq\frac{(\tr(B)-\la e_n,Ae_n\ra)^2}{n-1},
\end{aligned}
\end{equation}
where $I$ stands for the identity matrix in $\R^{(n-1)\times(n-1)}$. 
Note that if $n=2$, we have equality in place of inequality in the above display \eqref{eq:EstimateofSubmatrixofA}. 
Rewrite the middle term on the right hand side of \eqref{eq:BasicDecompositionofA} as
\begin{align} \label{eq:MiddleTermofARewritten}
2\sum_{i=1}^{n-1}\la e_i,Ae_n\ra^2
&=2|Ae_n|^2-2\la e_n,Ae_n\ra^2.
\end{align}
As we plug \eqref{eq:EstimateofSubmatrixofA} and \eqref{eq:MiddleTermofARewritten} into \eqref{eq:BasicDecompositionofA}, we obtain
\begin{align*}
|B|^2
&\geq\frac{(\tr(B)-\la e_n,Ae_n \ra)^2}{n-1}+2|Ae_n|^2
-\la e_n,Ae_n\ra^2.
\end{align*}
The desired estimate now follows, since by the cyclic property of trace we have $\tr(B)=\tr(A)$, and
$|B|^2=\tr(B^{\intercal}B)=\tr(A^{\intercal}A)=|A|^2$.
\end{proof}

\begin{corollary} \label{cor:OurVersionofFundamentalInequality}
If $v$ is a smooth function in a domain $\Omega\subset\Rn$, $n\geq 2$, then we have
\begin{equation} \label{eq:OurVersionofFundamentalInequality}
|Dv|^4|D^2v|^2
\geq 2|Dv|^2|D^2vDv|^2+\frac{\big(|Dv|^2\Delta v-\il v\big)^2}{n-1}
-(\il v)^2.
\end{equation}
everywhere in $\Omega$. If $n=2$, equality holds in the place of the inequality in \eqref{eq:OurVersionofFundamentalInequality}.
\end{corollary}

\begin{proof}
Let $A=D^2v$ and $e=Dv$ in \eqref{eq:MatrixInequality}.
\end{proof}

\subsection{Comparison between Corollary \ref{cor:OurVersionofFundamentalInequality} and the inequality (\ref{eq:FundamentalInequalityinDong2019})}
We rewrite the two inequalities given by \eqref{eq:FundamentalInequalityinDong2019} as two lower bounds for the quantity $|Dv|^2|D^2v|^2$. Thus \eqref{eq:FundamentalInequalityinDong2019} is equivalent with the two inequalities
\begin{equation} \label{eq:FundamentalInequalityinDong2019LHS}
(n-3)|Dv|^2|D^2v|^2
\geq (n-4)|D^2vDv|^2-|Dv|^2(\Delta v)^2+2\Delta v\il v 
\end{equation}
and
\begin{equation} \label{eq:FundamentalInequalityinDong2019RHS}
|Dv|^2|D^2v|^2
\geq \frac{n}{n-1}|D^2vDv|^2
+\frac{1}{n-1}|Dv|^2(\Delta v)^2-\frac{2}{n-1}\Delta v\il v.
\end{equation}
It is easy to show that the bound \eqref{eq:FundamentalInequalityinDong2019LHS} is trivial. We now compare \eqref{eq:FundamentalInequalityinDong2019RHS} with \eqref{eq:OurVersionofFundamentalInequality}, and show that \eqref{eq:OurVersionofFundamentalInequality} is slightly sharper. Namely, we rewrite \eqref{eq:OurVersionofFundamentalInequality} as
\begin{align*}
|Dv|^2|D^2v|^2 \geq
&\,\,\frac{n}{n-1}|D^2vDv|^2+\frac{1}{n-1}|Dv|^2(\Delta v)^2
-\frac{2}{n-1}\Delta v\il v \\
&+\frac{n-2}{n-1}\Big(|D^2vDv|^2-\frac{(\il v)^2}{|Dv|^2}\Big).
\end{align*}
By Cauchy-Schwartz inequality
$$ (\il v)^2 = \la Dv,D^2vDv\ra^2 \leq |Dv|^2|D^2vDv|^2. $$
Hence \eqref{eq:OurVersionofFundamentalInequality} implies \eqref{eq:FundamentalInequalityinDong2019}.

\section{Application of the inequality}

The following Theorem is an improved version of Theorem 1.1 in \cite{Dong2020}.

\begin{theorem} \label{thm:Application}
Let $n\geq 2$, $1<p<\infty$ and $s>-1-\frac{p-1}{n-1}$. If $u$ is $p$-harmonic in $\Omega\subset\Rn$, then $V_s(Du)\in W^{1,2}_{\loc}(\Omega)$. 
Moreover, there exists a constant $C=C(n,p,s)>0$ such that
\begin{equation} \label{eq:QuantitativeBoundforSecondDerivatives}
\int_{B_r}|D(V_s(Du))|^2dx
\leq \frac{C}{r^2}\int_{B_{2r}}|V_s(Du)|^2dx
\end{equation}
for any concentric balls $B_r\subset B_{2r}\subset\subset\Om$.
\end{theorem}

To prove Theorem \ref{thm:Application}, we use essentially the same proof as in \cite{Dong2020}. The only significant difference is that we apply the sharper inequality \eqref{eq:OurVersionofFundamentalInequality} in Corollary \ref{cor:OurVersionofFundamentalInequality} instead of the inequality \eqref{eq:FundamentalInequalityinDong2019}. For the reader's convenience, we provide a detailed proof of Theorem \ref{thm:Application}.

Let $u$ be $p$-harmonic in $\Omega\subset\Rn$ and $U\subset\subset\Om$ be a smooth subdomain of $\Om$. 
For $\epsilon>0$ small, consider the regularized Dirichlet problem
\begin{equation} \label{eq:RegularizedDirichletProblem}
\begin{cases}
\begin{aligned}
\diverg\Big((|D\ue|^2+\epsilon)^{\frac{p-2}{2}}D\ue\Big)=0 &\quad\text{in }U; \\
\ue=u &\quad\text{on }\partial U.
\end{aligned}
\end{cases}
\end{equation}
By the standard elliptic regularity theory \cite{GilbargTrudninger}, there exists a unique solution $\ue\in C^{\infty}(U)\cap C^0(\overline{U})$. Furthermore, the family $\{\ue\}_{\epsilon}$ is uniformly bounded in $C^{1,\alpha}_{\loc}(U)$ for some $\alpha=\alpha(n,p)\in(0,1)$. That is, for any subdomain $V\subset\subset U$
there exists a constant \\
$C=C(n,p,\dist(V,\partial U),\|u\|_{L^{\infty}(U)})>0$ such that
\begin{equation} \label{eq:UniformC1alphaBoundednessofRegularizedSolutions}
\|\ue\|_{C^{1,\alpha}(V)}\leq C,
\end{equation}
see for instance \cite{Wang1994}. The Arzel\`a-Ascoli compactness theorem implies that
\begin{equation} \label{eq:PointwiseConvergenceofRegularizedDerivative}
D\ue \xrightarrow{\epsilon\to0} Du \quad\text{locally uniformly in }U,
\end{equation}
up to a subsequence. Hereafter, we always consider appropriate subsequences of the family $\{\ue\}_{\epsilon}$.

For notational convenience, we introduce the regularized version of the vector field $V_s$. Let us define $\Ve_s\colon\Rn\to\Rn$ as
$$ \Ve_s(z):=(|z|^2+\epsilon)^{\frac{p-2+s}{4}}z \quad\text{for }z\in\Rn. $$
We aim to show a bound similar to \eqref{eq:QuantitativeBoundforSecondDerivatives} for $\Ve_s(D\ue)$. 
Namely, we show that there exists a constant $C=C(n,p,s)>0$ such that
\begin{equation} \label{eq:RegularizedQuantitativeBoundforSecondDerivatives}
\int_{U}|D(\Ve_s(D\ue))|^2\phi^2dx
\leq C\int_{U}(|D\ue|^2+\epsilon)^{\frac{p+s}{2}}|D\phi|^2dx
\end{equation}
for any $\phi\in\CcU$.

The estimate \eqref{eq:QuantitativeBoundforSecondDerivatives} can be derived from \eqref{eq:RegularizedQuantitativeBoundforSecondDerivatives} as follows.
Let us fix the concentric balls $B_r\subset B_{2r}\subset\subset \Om$ and select a subdomain $U\subset\subset\Om$ such that $B_{2r}\subset\subset U$. Let $\phi\in\CcU$ be a cutoff function such that
$$ \phi=1\quad\text{in }B_r, 
\quad \spt \phi = \overline{B}_{2r}
\quad{\text{and}}\quad |D\phi|\leq \frac{10}{r}. $$
The estimate \eqref{eq:RegularizedQuantitativeBoundforSecondDerivatives} implies that
\begin{equation} \label{eq:RegularizedQuantitativeBoundforSecondDerivativesinBall}
\int_{B_r}|D(\Ve_s(D\ue))|^2dx
\leq \frac{C}{r^2}\int_{B_{2r}}(|D\ue|^2+\epsilon)^{\frac{p+s}{2}}dx
\end{equation}
for $C=C(n,p,s)>0$. If $s>-p$, we can apply \eqref{eq:UniformC1alphaBoundednessofRegularizedSolutions} to conclude that the right hand side of the above display \eqref{eq:RegularizedQuantitativeBoundforSecondDerivativesinBall} is bounded from above by a constant independent of $\epsilon$. Thus $\{\Ve_s(D\ue)\}_{\epsilon}$ is bounded in $W^{1,2}(B_r)$, and consequently we may extract a subsequence that converges weakly in $W^{1,2}(B_r)$ and strongly in $L^q(B_r)$ for any $1\leq q<\frac{2n}{n-2}$. 
By \eqref{eq:PointwiseConvergenceofRegularizedDerivative} and Dominated convergence theorem
\begin{equation}
\int_{B_{2r}}(|D\ue|^2+\epsilon)^{\frac{p+s}{2}}dx 
\xrightarrow{\epsilon\to0} 
\int_{B_{2r}}|V_s(Du)|^2dx
\end{equation}
and
\begin{equation}
\Ve_s(D\ue)\xrightarrow{\epsilon\to0}  V_s(Du) \quad\text{in }L^2(B_r).
\end{equation}
Finally, recalling that norm is lower semicontinuous with respect to the weak convergence, we can let $\epsilon\to0$ in \eqref{eq:RegularizedQuantitativeBoundforSecondDerivativesinBall} to obtain \eqref{eq:QuantitativeBoundforSecondDerivatives}. 
\subsection{Caccioppoli type estimates}
Let us henceforth denote
$$ \mu:=(|D\ue|^2+\epsilon)^{1/2} $$
and
$$ A:=I+(p-2)\frac{D\ue\otimes D\ue}{|D\ue|^2+\epsilon}, $$
where $I$ stands for the identity matrix in $\R^{n\times n}$ and $\otimes$ stands for the tensor product (or outer product) of two vectors in $\Rn$, resulting a matrix in $\R^{n\times n}$.
Note that
\begin{equation} \label{eq:EllipticityofA}
\min \{1, p-1\}I \leq A \leq \max \{1,p-1\} I
\end{equation}
uniformly in $U$.
Differentiating the PDE in \eqref{eq:RegularizedDirichletProblem} yields that the partial derivatives $\ue_{x_k}$, $k=1,\ldots,n$ solve the linear, degenerate elliptic equation
\begin{equation} \label{eq:RegularizedLinearization}
\diverg\Big(\mu^{p-2}AD\ue_{x_k}\Big)=0.
\end{equation}
In this subsection we test the equation \eqref{eq:RegularizedLinearization} with various test functions. 

The following Lemma is the basic Caccioppoli type estimate related to the equation \eqref{eq:RegularizedLinearization}. It will not be needed to prove Theorem \ref{thm:Application}. Instead, it will be employed in Section 4.

\begin{lemma} \label{lem:CaccioppoliVeryBasic}
Let $\ue$ solve \eqref{eq:RegularizedDirichletProblem}. Then we have for any $\phi\in\CcU$ and $z\in\Rn$ that
\begin{equation} \label{eq:CaccioppoliVeryBasic}
\int_U \mu^{p-2}|D^2\ue|^2\phi^2dx 
\leq C\int_U\mu^{p-2}|D\ue-z|^2|D\phi|^2dx,    
\end{equation}
where $C=C(p)>0$ is independent of $\epsilon$.
\end{lemma}

\begin{proof}
Let $\phi\in\CcU$ and $z=(z_1,\ldots, z_n)\in\Rn$ and put
$$ \varphi=\phi^2(\ue_{x_k}-z_k). $$
We have
$$ D\varphi = 2\phi(\ue_{x_k}-z_k)D\phi + \phi^2D\ue_{x_k}, $$
and hence
\begin{align*}
\int_U\mu^{p-2}\la AD\ue_{x_k},D\ue_{x_k}\ra \phi^2dx 
&= -2\int_U\mu^{p-2}\la AD\ue_{x_k},D\phi\ra(\ue_{x_k}-z_k)\phi dx \\
&\leq 2\int_U\mu^{p-2} \sqrt{\la A D\ue_{x_k},D\ue_{x_k}\ra}\sqrt{\la A D\phi,D\phi\ra}
|\ue_{x_k}-z_k||\phi|dx.
\end{align*}
Application of Young's inequality together with the uniform ellipticity of $A$, \eqref{eq:EllipticityofA}, yields
$$ \int_U\mu^{p-2}|D\ue_{x_k}|^2\phi^2 dx 
\leq C\int_U\mu^{p-2}|D\phi|^2|\ue_{x_k}-z_k|^2dx, $$
where $C=C(p)>0$. Finally sum over $k=1,\ldots,n$ to conclude \eqref{eq:CaccioppoliVeryBasic}.
\end{proof}

The following Lemma is analogous to Lemma 3.1 in \cite{Dong2020}.

\begin{lemma} \label{lem:CaccioppoliEstimate}
Let $\ue$ solve \eqref{eq:RegularizedDirichletProblem} and let $s\in\R$. Then we have for any $\eta>0$ and for any $\phi\in\CcU$ that
\begin{equation} \label{eq:CaccioppoliEstimate}
\begin{aligned}
&\int_U |D^2\ue|^2\mu^{p-2+s}\phi^2dx 
+ (p-2+s-\eta)\int_U|D^2\ue D\ue|^2\mu^{p-4+s}\phi^2dx \\
&\quad\quad  + (s(p-2)-\eta)\int_U (\il\ue)^2 \mu^{p-6+s}\phi^2dx 
\leq \frac{C}{\eta}\int_U\mu^{p+s}|D\phi|^2dx,
\end{aligned}
\end{equation}
where $C=C(p)>0$ is independent of $\epsilon$.
\end{lemma}

\begin{proof}
Let $\phi\in\CcU$  and $s\in\R$, and put
$$ \varphi=\phi^2\mu^{s}\ue_{x_k}. $$
We have
\begin{align*}
D\varphi
&=2\phi\mu^{s}\ue_{x_k}D\phi
+s\mu^{s-2}\phi^2\ue_{x_k}D^2\ue D\ue 
+\phi^2\mu^{s}D\ue_{x_k}.
\end{align*}
To ease the notation, let
$$ w:=\mu^{p-2+s}\phi^2. $$
We obtain
\begin{equation} \label{eq:CaccioppoliStartingPoint}
\begin{aligned}
&\int_U\la AD\ue_{x_k},D\ue_{x_k}\ra wdx 
+ s\int_U \mu^{-2}\la AD\ue_{x_k},D^2\ue D\ue\ra\ue_{x_k} wdx \\
&\quad = -2\int_U\la AD\ue_{x_k},D\phi\ra\ue_{x_k} \phi^{-1}wdx.
\end{aligned}    
\end{equation}
Note that
$$ \la AD\ue_{x_k},D\ue_{x_k}\ra 
= |D\ue_{x_k}|^2+(p-2)\frac{\la D\ue,D\ue_{x_k}\ra^2}{\mu^2} $$
and
\begin{align*}
\la AD\ue_{x_k},D^2\ue D\ue\ra 
= \la D\ue_{x_k},D^2\ue D\ue\ra 
+(p-2)\frac{\la D\ue_{x_k},D\ue\ra\il\ue}{\mu^2}.
\end{align*}
Summing over $k=1,\ldots,n$ yields
\begin{equation} \label{eq:CaccioppoliIdentity}
\begin{aligned}
&\int_U |D^2\ue|^2 wdx 
+ (p-2+s)\int_U\mu^{-2}|D^2\ue D\ue|^2wdx \\
&\quad + s(p-2)\int_U \mu^{-4}(\il\ue)^2 wdx = -2\int_U\la AD^2\ue D\ue,D\phi\ra \phi^{-1}wdx.
\end{aligned}
\end{equation}
The proof follows from the identity \eqref{eq:CaccioppoliIdentity} via an application of Young's inequality.
For any $\eta>0$, we can estimate the integrand on the right hand side of \eqref{eq:CaccioppoliIdentity} as follows:
\begin{align*}
   -2\la AD^2\ue D\ue,D\phi\ra \phi^{-1}w
   &\leq 2|D^2\ue D\ue||D\phi|\phi^{-1}w
   +2|p-2|\frac{|\il\ue||D\ue||D\phi|}{\mu^2}\phi^{-1}w \\
   &\leq \eta|D^2\ue D\ue|^2\mu^{-2}w
   +\frac{C}{\eta}|D\phi|^2\mu^2\phi^{-2}w \\
   &\quad +\eta(\il\ue)^2\mu^{-4}w 
   +\frac{C(p-2)^2}{\eta}|D\ue|^2|D\phi|^2\phi^{-2}w,
\end{align*}
where $C>0$ is an absolute constant. The proof is complete.
\end{proof}

The following Corollary gives, roughly speaking, an $L^2$-estimate for the Hessian $D^2\ue$ in terms of the second order derivative quantity $D^2\ue D\ue$ and the gradient $D\ue$.

\begin{corollary} \label{cor:CaccioppoliEstimateforHessian}
Let $\ue$ solve \eqref{eq:RegularizedDirichletProblem} and let $s\in\R$. Then we have for any $\phi\in\CcU$ that
\begin{equation}
\int_U|D^2\ue|^2\mu^{p-2+s}\phi^2dx
\leq 
C\Big(
\int_U|D^2\ue D\ue|^2\mu^{p-4+s}\phi^2dx
+\int_U\mu^{p+s}|D\phi|^2dx
\Big)
\end{equation}
where $C=C(p,s)>0$ is independent of $\epsilon$.
\end{corollary}

\begin{proof}
Move the second and third integral on the left hand side of \eqref{eq:CaccioppoliEstimate} to the right hand side of the inequality. Estimate
$$ (\il\ue)^2 \leq |D\ue|^2|D^2\ue D\ue|^2\leq \mu^2|D^2\ue D\ue|^2 $$
to conclude the proof.
\end{proof}


\subsection{Lower bound for $|D^2\ue|^2$ and proof of Theorem \ref{thm:Application}}

We begin with observing that by the smoothness of $\ue$, $|D\ue|$ is locally Lipschitz continuous, and thus, by Rademacher theorem, differentiable almost everywhere. Moreover, if $D\ue=0$ at a point where $|D\ue|$ is differentiable, we must have $D|D\ue|=0$ at that point. This allows us to define the normalized infinity Laplacian
$$ \ilN\ue :=\la \frac{D\ue}{|D\ue|},D|D\ue|\ra  $$
almost everywhere in $U$.
Note that if $D\ue\neq0$, we have 
$$ \ilN\ue =\frac{\il\ue}{|D\ue|^2}. $$
We can therefore rewrite
\begin{equation} \label{eq:ReformulationofDerivativeQuantities}
|D^2\ue D\ue|^2=|D\ue|^2|D|D\ue||^2
\quad\text{and}\quad
(\il\ue)^2=|D\ue|^4(\ilN\ue)^2
\end{equation}
almost everywhere in $U$.

\begin{lemma} \label{lem:RegularizedLowerBoundforHessian}
Let $n\geq 2$ and $\ue$ solve \eqref{eq:RegularizedDirichletProblem}. Then
\begin{equation} \label{eq:RegularizedLowerBoundforHessian}
|D^2\ue|^2
\geq 2|D|D\ue||^2+\Phi(\ilN\ue)^2
\end{equation} 
almost everywhere in $U$, where
$$ \Phi:
=\frac{(p-1)^2}{n-1}-1
-\frac{\epsilon}{\mu^2}\cdot\frac{2(p-1)(p-2)}{n-1}
+\frac{\epsilon^2}{\mu^4}\cdot\frac{(p-2)^2}{n-1}. $$
If $n=2$, equality holds in the place of inequality in \eqref{eq:RegularizedLowerBoundforHessian}.
\end{lemma}

\begin{proof}
By the smoothness of $\ue$, the non-divergence form of the PDE in \eqref{eq:RegularizedDirichletProblem},
\begin{equation} \label{eq:RegularizedNonDivergenceForm}
\Delta\ue+(p-2)\frac{\il\ue}{|D\ue|^2+\epsilon}=0,
\end{equation}
is equivalent with the original one.
The proof follows now immediately from Corollary \ref{cor:OurVersionofFundamentalInequality}, by plugging the non-divergence form \eqref{eq:RegularizedNonDivergenceForm} into \eqref{eq:OurVersionofFundamentalInequality}.
\end{proof}

Finally we gather together the above estimates to prove Theorem \ref{thm:Application}.

\begin{proof}[Proof of Theorem \ref{thm:Application}]
Recall that to prove Theorem \ref{thm:Application} it suffices to show that the estimate \eqref{eq:RegularizedQuantitativeBoundforSecondDerivatives}, that is,
\begin{equation*}
\int_{U}|D(\mu^{\frac{p-2+s}{2}}D\ue)|^2\phi^2dx
\leq C\int_{U}\mu^{p+s}|D\phi|^2dx,
\end{equation*}
holds for any $\phi\in C^{\infty}_0(U)$ with a constant $C=C(n,p,s)>0$ independent of $\epsilon$.
We start with 
\begin{align*}
\int_U|D(\mu^{\frac{p-2+s}{2}}D\ue)|^2\phi^2dx
&=\int_U\mu^{p-2+s}\Big(|D^2\ue|^2
+(p-2+s)\frac{|D^2\ue D\ue|^2}{\mu^2} \\
&\quad\,\, +\frac{(p-2+s)^2}{4}\frac{|D\ue|^2|D^2\ue D\ue|^2}{\mu^4}\Big)\phi^2dx \\
&\leq \Big(1+|p-2+s|
+\frac{(p-2+s)^2}{4}\Big)\int_U|D^2\ue|^2\mu^{p-2+s}\phi^2dx.
\end{align*}
We apply Corollary \ref{cor:CaccioppoliEstimateforHessian} to obtain
\begin{equation} \label{eq:RegularizedEstimateforDVsinTermsofD2uDuandDu}
\begin{aligned} 
\int_U|D(\mu^{\frac{p-2+s}{2}}D\ue)|^2\phi^2dx 
\leq C(p,s)\Big(
\int_U|D^2\ue D\ue|^2\mu^{p-4+s}\phi^2dx \\
+\int_U\mu^{p+s}|D\phi|^2dx\Big).
\end{aligned}    
\end{equation}
This estimate holds for any $s\in\R$.

In the remaining part of the proof, we estimate the first integral on the right hand side of \eqref{eq:RegularizedEstimateforDVsinTermsofD2uDuandDu} by combining Lemma \ref{lem:CaccioppoliEstimate} and Lemma \ref{lem:RegularizedLowerBoundforHessian}. We estimate the first integral of the left hand side of \eqref{eq:CaccioppoliEstimate} from below by \eqref{eq:RegularizedLowerBoundforHessian}. In addition we rewrite $|D^2\ue D\ue|^2$ and $\il\ue$ on the left hand side of \eqref{eq:CaccioppoliEstimate} according to \eqref{eq:ReformulationofDerivativeQuantities}.
We conclude that for any $\eta>0$ and for any $\phi\in\CcU$
\begin{align*}
&\int_U \Big((p-2+s-\eta)\frac{|D\ue|^2}{\mu^2}+2\Big)|D|D\ue||^2 wdx  \\
&+ \int_U\Big(\Phi+ (s(p-2)-\eta)\frac{|D\ue|^4}{\mu^4}\Big)(\ilN\ue)^2 wdx \\
&\quad \leq \frac{C}{\eta}\int_U\mu^{p+s}|D\phi|^2dx,
\end{align*}
where $C=C(p)>0$ and $w:=\mu^{p-2+s}\phi^2$.

Writing
$$ 1=\frac{|D\ue|^2}{\mu^2}+\frac{\epsilon}{\mu^2} $$
yields
\begin{align*}
&\int_U \Big((p+s-\eta)\frac{|D\ue|^2}{\mu^2}
+\frac{2\epsilon}{\mu^2}\Big)|D|D\ue||^2 wdx
+\int_U\Psi(\ilN\ue)^2 wdx \\
&\quad \leq \frac{C}{\eta}\int_U\mu^{p+s}|D\phi|^2dx,
\end{align*}
where
\begin{align*}
\Psi:=\Phi+(s(p-2)-\eta)\frac{|D\ue|^4}{\mu^4}.
\end{align*}
Observe that, if $s>-1-\frac{p-1}{n-1}$ then also $s>-p$, and we may choose $\eta=\eta(p,s)>0$ so small that we can estimate
\begin{align*}
&\eta\int_U\frac{|D\ue|^2}{\mu^2}|D|D\ue||^2wdx
+\int_U \Big((p+s-2\eta)\frac{|D\ue|^2}{\mu^2}
+\frac{2\epsilon}{\mu^2}+\Psi\Big)(\ilN\ue)^2 wdx \\
&\quad \leq \frac{C}{\eta}\int_U\mu^{p+s}|D\phi|^2dx.
\end{align*}
Now it remains to show that the condition $s>-1-\frac{p-1}{n-1}$ guarantees that we can adjust $\eta>0$ even further so that
$$ \int_U \Big((p+s-2\eta)\frac{|D\ue|^2}{\mu^2}
+\frac{2\epsilon}{\mu^2}+\Psi\Big)(\ilN\ue)^2 wdx\geq 0. $$
Note that
\begin{align*}
(p+s-2\eta)\frac{|D\ue|^2}{\mu^2}
+\frac{2\epsilon}{\mu^2}+\Psi
=a\frac{|D\ue|^4}{\mu^4}
+b\frac{\epsilon|D\ue|^2}{\mu^2}
+c\frac{\epsilon^2}{\mu^4}
\end{align*}
where
\begin{align*}
a
&=(p-1)\Big(s+1+\frac{p-1}{n-1}\Big)-3\eta,
\end{align*}
\begin{align*}
b
&=p+s +\frac{2(p-1)^2}{n-1}-\frac{2(p-1)(p-2)}{n-1}-2\eta \\
&=p+s+\frac{2(p-1)}{n-1}-2\eta
\end{align*}
and
\begin{align*}
c
&=1+\frac{(p-1)^2}{n-1}-\frac{2(p-1)(p-2)}{n-1}+\frac{(p-2)^2}{n-1} \\
&=1+\frac{1}{n-1}.
\end{align*}
We can now easily see that the restrictive condition for $s$ is indeed $s>-1-\frac{p-1}{n-1}$.
\end{proof}

\section{Proof of Theorem \ref{thm:OscillationEstimateforDVs}}
In this section we explain how to conclude the estimate \eqref{eq:OscillationEstimateforDVs} from the estimate \eqref{eq:QuantitativeBoundforSecondDerivatives} 
and the known regularity results of $p$-harmonic functions. First, note that it suffices to find $C=C(n,p,s)>0$ and $M=M(n,p,s)\geq 4$ such that
\begin{equation} \label{eq:OscillationEstimateforDVsMversion}
\int_{B_r}|D(V_s(Du))|^2dx
\leq \frac{C}{r^2}\int_{B_{Mr}}|V_s(Du)-z|^2dx
\end{equation}
for all vectors $z\in\Rn$ and all concentric balls $B_r\subset B_{Mr}\subset\subset \Omega$. Indeed, fix $B_r\subset B_{2r}\subset\subset\Om$ concentric and let $\rho:=M^{-1}r$. There exists an integer $N=N(n,M)>0$ such that
$B_r$ may be covered with a family $\{B_{\rho}(x_i)\}_{i=1}^N$, where the center points $x_i\in B_r$. Then
\begin{align*}
\int_{B_r}|D(V_s(Du))|^2dx
&\leq \sum_{i=1}^N\int_{B_{\rho}(x_i)}|D(V_s(Du))|^2dx 
\leq \sum_{i=1}^N\frac{C}{\rho^2}\int_{B_{M\rho}(x_i)}|V_s(Du)-z|^2dx \\
&\leq \frac{CNM^2}{r^2}\int_{B_{2r}}|V_s(Du)-z|^2dx.
\end{align*}
Also, note that it suffices to show \eqref{eq:OscillationEstimateforDVsMversion} for $z=\big(V_s(Du)\big)_{B_{Mr}}$.

To show \eqref{eq:OscillationEstimateforDVsMversion} for some $M\geq 4$ to be selected later, we divide the sufficiently small balls $B_r$ inside $\Om$ into two categories. By 'sufficiently', we mean that $B_{Mr}\subset\subset\Om$. In our setting, we say a ball $B_r\subset\subset\Om$ is degenerate if
\begin{equation} \label{eq:DegenracyCondition}
        \int_{B_{2r}}|V_s(Du)|^2\leq \int_{B_{Mr}}|V_s(Du)-\big(V_s(Du)\big)_{B_{Mr}}|^2dx;
\end{equation}
and non-degenerate if
\begin{equation} \label{eq:NonDegeneracyCondition}
        \int_{B_{2r}}|V_s(Du)|^2> \int_{B_{Mr}}|V_s(Du)-\big(V_s(Du)\big)_{B_{Mr}}|^2dx.
\end{equation}
In this section such balls $B_r$, $B_{2r}$ and $B_{Mr}$ are always assumed to be concentric unless otherwise stated.

Let us fix a ball $B_r$ such that $B_{Mr}\subset\subset\Om$. The ball $B_r$ must be either degenerate of non-degenerate. If $B_r$ is degenerate,  then \eqref{eq:OscillationEstimateforDVsMversion} follows directly from \eqref{eq:QuantitativeBoundforSecondDerivatives}. In this case we need to restrict $s>-1-\frac{p-1}{n-1}$. If $B_r$ is non-degenerate, 
we apply a method from the proof of Proposition 5.1 in \cite{Avelin2018}. The main consequence of the non-degeneracy condition \eqref{eq:NonDegeneracyCondition} is that we can select $M$ so large that $Du$ is approximately a nonzero constant vector in $B_{2r}$. To prove this we use the known $C^{1,\alpha}_{\loc}$-regularity of $p$-harmonic functions. We remark that in the non-degenerate case it suffices to restrict $s>-p$. If $n=2$, the degenerate and non-degenerate conditions for $s$ are the same.

The following Theorem summarizes the basic regularity of $p$-harmonic functions that we need to prove Theorem \ref{thm:OscillationEstimateforDVs}. For the proof we refer to \cite{Lewis1983} and \cite{Iwaniec1985}*{Theorem 2}, \cite{Mingione2007}*{Lemma 3.1}.
\begin{theorem} \label{thm:RegularityofpHarmonicFunctions}
Let $n\geq 2$ and $1<p<\infty$. There exists $\alpha=\alpha(n,p)\in(0,1)$ such that any $p$-harmonic function $u$ in $\Om\subset\Rn$ belongs to $C^{1,\alpha}_{\loc}(\Om)$.
Moreover, for any fixed $t>0$, there exists a constant $C=C(n,p,t)>0$ such that
\begin{equation}
\underset{B_r}{\osc} Du 
\leq C\Big(\frac{r}{R}\Big)^{\alpha}\Big(\fint_{B_R}|Du|^tdx\Big)^{1/t}
\end{equation}
holds for all concentric balls $B_r\subset B_{2r}\subset B_{R}\subset\subset\Om$.
\end{theorem}

The following lemma is a straightforward generalization of Lemma 5.3 in \cite{Avelin2018}.

\begin{lemma} \label{lem:ChangeofScale}
Let $\Om\subset\Rn$ and $v\in L^2_{\loc}(\Om)$ be such that
\begin{equation} \label{eq:AssumptionofChangeofScale}
   \int_{B_{mr}}|v|^2dx > \int_{B_{Mr}}|v-(v)_{B_{Mr}}|^2dx 
\end{equation}
for some concentric balls $B_{mr}\subset B_{Mr}\subset\subset\Om$, where $0<m<M<\infty$. Then for any $\kappa\in[m,M]$ we have
$$ \fint_{B_{\kappa r}}|v|^2dx
\leq 9\fint_{B_{mr}}|v|^2dx. $$
\end{lemma}

\begin{proof}
Apply Minkowski inequality and then H\"older inequality to obtain
\begin{align*}
    \Big(\fint_{B_{\kappa r}}|v|^2\Big)^{1/2}
    &\leq \Big(\fint_{B_{\kappa r}}|v-(v)_{B_{Mr}}|^2dx\Big)^{1/2} 
    + \fint_{B_{mr}}|v-(v)_{B_{Mr}}|dx
    + \fint_{B_{mr}}|v|dx \\
    &\leq \Big(\fint_{B_{\kappa r}}|v-(v)_{B_R}|^2dx\Big)^{1/2} 
    + \Big(\fint_{B_{mr}}|v-(v)_{B_{Mr}}|^2dx\Big)^{1/2}
    + \Big(\fint_{B_{mr}}|v|^2dx\Big)^{1/2}.
\end{align*}
Enlarging the integral domains in the first two items on the bottom row of the above display yields
\begin{align*}
    \Big(\fint_{B_{\kappa r}}|v|^2\Big)^{1/2}
    &\leq 2\Big(|B_{mr}|^{-1}\int_{B_{Mr}}|v-(v)_{B_{Mr}}|^2dx\Big)^{1/2}
    + \Big(\fint_{B_{mr}}|v|^2dx\Big)^{1/2}.
\end{align*}
Now the assumption \eqref{eq:AssumptionofChangeofScale} is applicable on the first item on the right hand side of the above inequality. The desired estimate follows and the proof is complete.
\end{proof}

For the proof of the following algebraic inequalities, see \cite{Hamburger1992}*{Lemma 2.1}. 

\begin{lemma} \label{lem:AlgebraicIdentity}
Let $1<p<\infty$ and $s>-p$. There exist constants $c_1=c_1(p,s)>0$ and $c_2=c_2(p,s)>0$ such that
$$ c_1\big(\epsilon+|z|^2+|w|^2\big)^{\frac{p-2+s}{2}}|z-w|^2 
\leq \big|\Ve_s(z)-\Ve_s(w)\big|^2 \leq c_2 \big(\epsilon+|z|^2+|w|^2\big)^{\frac{p-2+s}{2}}|z-w|^2 $$
for any two vectors $z,w\in\Rn$.
\end{lemma}

Let us introduce the notation
$$ \lambda:
=\Big(\fint_{B_{2r}}|Du|^{p+s}dx\Big)^{\frac{1}{p+s}}
=\Big(\fint_{B_{2r}}|V_s(Du)|^2dx\Big)^{\frac{1}{p+s}}. $$
Note that if $\lambda=0$, then the desired estimate \eqref{eq:OscillationEstimateforDVs} is trivial. Hence we may assume that $\lambda>0$.

The following lemma is an adapted version of Lemma 5.5 in \cite{Avelin2018}. 
\begin{lemma} \label{lem:GradientComparability}
Let $n\geq 2$, $1<p<\infty$ and $s>-p$. Suppose that $u$ is $p$-harmonic in $\Om\subset\Rn$. Given any $\sigma>0$, there exists a constant $M=M(n,p,s,\sigma)\geq 4$ such that for any ball $B_r\subset\subset\Om$  the non-degeneracy condition \eqref{eq:NonDegeneracyCondition} implies that
\begin{equation} \label{eq:GradientComparability}
|Du-Du(x_0)|\leq \sigma\lambda \quad\text{in }B_{2r},
\end{equation}
where $x_0\in B_{2r}$ is a point such that $|Du(x_0)|=\lambda.$
\end{lemma}

\begin{proof}
By mean value theorem, we can fix a point $x_0\in B_{2r}$ such that
$|Du(x_0)|=\lambda.$
Let $x\in B_{2r}$. We apply Theorem \ref{thm:RegularityofpHarmonicFunctions} with $t=p+s>0$ to estimate
$$ |Du(x)-Du(x_0)|\leq\underset{B_{2r}}{\osc} Du
\leq C\Big(\frac{2}{M}\Big)^{\alpha}\Big(\fint_{B_{Mr}}|Du|^{p+s}dx\Big)^{\frac{1}{p+s}}, $$
where $C=C(n,p,s)>0$.
The non-degeneracy condition \eqref{eq:NonDegeneracyCondition} allows us to employ Lemma \ref{lem:ChangeofScale} with $v=V_s(Du)$ and $m=2$ to obtain
$$ \Big(\fint_{B_{Mr}}|Du|^{p+s}dx\Big)^{\frac{1}{p+s}}
\leq 9^{\frac{1}{p+s}}\lambda. $$
We can now adjust $M=M(n,p,s,\sigma)\geq 4$ such that
$C\big(\frac{2}{M}\big)^{\alpha} 9^{\frac{1}{p+s}}\leq\sigma$.  
This completes the proof.
\end{proof}

We are finally ready to complete the proof of Theorem \ref{thm:OscillationEstimateforDVs}.

\begin{proof}[Proof of Theorem \ref{thm:OscillationEstimateforDVs}]
Let $\sigma=\sigma(p,s)>0$ be a very small constant to be selected later, and accordingly let $M=M(n,p,s,\sigma)\geq 4$ be given by Lemma \ref{lem:GradientComparability}.
Fix a ball $B_r\subset\subset\Om$ such that $B_{Mr}\subset\subset\Om$. 
Recall that, in view of Theorem \ref{thm:Application}, it suffices to study the case when $B_r$ is non-degenerate \eqref{eq:NonDegeneracyCondition}. To run the computations, we consider the regularization \eqref{eq:RegularizedDirichletProblem} in a subdomain $U\subset\subset\Om$ such that $B_{Mr}\subset\subset U$. By \eqref{eq:PointwiseConvergenceofRegularizedDerivative} and Lemma \ref{lem:GradientComparability}, we may henceforth consider $0<\epsilon<\sigma\lambda^2$ so small that
\begin{equation} \label{eq:RegularizedGradientComparability}
|D\ue-Du(x_0)|\leq 2\sigma\lambda
\quad\text{and}\quad
\frac{3}{4}\lambda\leq \mu\leq \frac{5}{4}\lambda \quad\text{in }B_{2r},
\end{equation}
where $x_0\in B_{2r}$ is a point such that $|Du(x_0)|=\lambda$, and $\mu=(|D\ue|^2+\epsilon)^{1/2}$.

In what follows, the constants $C=C(n,p,s)>0$ and $c=c(p,s)>0$ may vary from line to line.
By \eqref{eq:RegularizedGradientComparability}
\begin{equation} \label{eq:RegularizedStartingEsimateforDVs}
|D(\Ve_s(D\ue))|^2\leq C\mu^{p-2+s}|D^2\ue|^2\leq C\lambda^s\mu^{p-2}|D^2\ue|^2 \quad\text{in }B_{2r}.
\end{equation}
We employ Lemma \ref{lem:CaccioppoliVeryBasic} with a cutoff function $\phi\in\CcU$ such that
$$ \phi=1\quad\text{in }B_r, 
\quad \spt \phi = \overline{B}_{2r}
\quad{\text{and}}\quad |D\phi|\leq \frac{10}{r}, $$
and use the estimates \eqref{eq:RegularizedStartingEsimateforDVs} and \eqref{eq:RegularizedGradientComparability}
to arrive at
\begin{equation} \label{eq:FirstL2EstimateofDVs}
\begin{aligned}
\int_{B_r}|D(\Ve_s(D\ue))|^2dx
& \leq \frac{C\lambda^s}{r^2}\int_{B_{2r}}\mu^{p-2}|D\ue-z|^2dx \\
& \leq \frac{C}{r^2}\int_{B_{2r}}\mu^{p-2+s}|D\ue-z|^2dx 
\end{aligned}
\end{equation}
for any $z\in\Rn$. 
In particular, since $\Ve_s\colon\Rn\to\Rn$ is bijective, we may select $z=\ze\in\Rn$ such that
\begin{equation}
\Ve_s(\ze)=\big(\Ve_s(D\ue)\big)_{B_{2r}}.
\end{equation}
Observe that, by Lemma \ref{lem:AlgebraicIdentity} and \eqref{eq:RegularizedGradientComparability},
\begin{equation} \label{eq:SmallnessinTermsofsigma}
\begin{aligned}
|\Ve_s(\ze)-\Ve_s(Du(x_0))|
&\leq \fint_{B_{2r}}|\Ve_s(D\ue)-\Ve(Du(x_0))|dx \\
&\leq c_2\fint_{B_{2r}}(\mu^2+\lambda^2)^{\frac{p-2+s}{4}}|D\ue-Du(x_0)|dx \\
&\leq c\sigma\lambda^{\frac{p+s}{2}}.
\end{aligned}    
\end{equation}
We employ the above estimate \eqref{eq:SmallnessinTermsofsigma} to estimate $|\ze|$ from above and below. 
If $p-2+s\geq0$, we have
\begin{equation*}
(1-c\sigma)\lambda^{\frac{p+s}{2}}
\leq (|\ze|^2+\epsilon)^{\frac{p-2+s}{4}}|\ze|
\leq ((1+\sigma)^{\frac{p-2+s}{4}}+c\sigma )\lambda^{\frac{p+s}{2}}.
\end{equation*}
If $p-2+s<0$, we have similarly
\begin{equation*}
((1+\sigma)^{\frac{p-2+s}{4}}-c\sigma)\lambda^{\frac{p+s}{2}}
\leq (|\ze|^2+\epsilon)^{\frac{p-2+s}{4}}|\ze|
\leq (1+c\sigma )\lambda^{\frac{p+s}{2}}.
\end{equation*}
Consequently, we may select $\sigma=\sigma(p,s)>0$ such that
$$ \frac{1}{2}\lambda^{\frac{p+s}{2}}
\leq (|\ze|^2+\epsilon)^{\frac{p-2+s}{4}}|\ze|
\leq 2\lambda^{\frac{p+s}{2}}. $$
We can now restrict $\epsilon$ so small, depending on $\lambda$, $p$ and $s$, that
\begin{equation} \label{eq:Comparabilityofze}
c^{-1}\lambda<|\ze|\leq c\lambda    
\end{equation}
for some $c=c(p,s)>0$.

We apply \eqref{eq:RegularizedGradientComparability} and \eqref{eq:Comparabilityofze}, together with Lemma \ref{lem:AlgebraicIdentity}, to estimate the integrand on the bottom row of \eqref{eq:FirstL2EstimateofDVs} with $z=\ze$ as follows;
\begin{equation} \label{eq:SecondIntegrandEstimateofDVs}
\begin{aligned}
\mu^{p-2+s}|D\ue-\ze|^2
&\leq c(\mu^2+|\ze|^2)^{\frac{p-2+s}{2}}|D\ue-\ze|^2dx \\
&\leq c|\Ve_s(D\ue)-\Ve(\ze)|^2 \quad\text{in }B_{2r}.
\end{aligned}
\end{equation}
Combination of \eqref{eq:FirstL2EstimateofDVs} and \eqref{eq:SecondIntegrandEstimateofDVs} yields that
\begin{equation} \label{eq:RegularizedOscillationEstimateforDVs}
\int_{B_r}|D(\Ve_s(D\ue))|^2dx
\leq \frac{C}{r^2}\int_{B_{2r}}|\Ve_s(D\ue)-\big(\Ve_s(D\ue)\big)_{B_{2r}}|^2dx,
\end{equation}
where $C=C(n,p,s)>0$ is independent of $\epsilon$.
Therefore, as explained in Section 3, we can let $\epsilon\to 0$ in \eqref{eq:RegularizedOscillationEstimateforDVs} to obtain
\begin{equation} 
\int_{B_r}|D(V_s(Du))|^2dx
\leq \frac{C}{r^2}\int_{B_{2r}}|V_s(Du)-\big(V_s(Du)\big)_{B_{2r}}|^2dx.
\end{equation}
Note that this implies \eqref{eq:OscillationEstimateforDVsMversion}. The proof is complete.
\end{proof}

\bibliographystyle{amsplain}
\bibliography{bibliography}

\end{document}